\documentclass[12pt]{article}
\usepackage{amsmath,amssymb,amsthm}

\usepackage{graphicx}
\usepackage{subfigure}

\newtheorem{theorem}{Theorem}[section]
\newtheorem{lemma}[theorem]{Lemma}

\newtheorem{corollary}[theorem]{Corollary}

\theoremstyle{definition}
\newtheorem{definition}[theorem]{Definition}
\numberwithin{equation}{section}
\numberwithin{figure}{section}
\newcommand{\Var}{\operatorname{Var}}

\newcommand{\Area}{\operatorname{Area}}

\begin{document}

\title{Central limit theorems for uniform model random polygons}

\author{John Pardon}

\date{2 June 2010; Revised 13 November 2010}

\maketitle

\begin{abstract}
We show how a central limit theorem for Poisson model random polygons implies a central limit theorem 
for uniform model random polygons.  To prove this implication, it suffices to show that in the two models, 
the variables in question have asymptotically the same expectation and variance.  We use integral geometric 
expressions for these expectations and variances to reduce the desired estimates to the convergence $(1+\frac\alpha n)^n\to e^\alpha$ as $n\to\infty$.
\end{abstract}

\section{Introduction}

Given a convex set $K\subset\mathbb R^2$ of unit area, we may define two random polygon models.  For the \emph{Poisson model}, 
we consider a Poisson process of intensity $\lambda$ inside $K$, and define $\Pi_{K,\lambda}$ to be the convex hull of the 
points of this process.  For the \emph{uniform model}, we take $n$ independent random points distributed uniformly in $K$ and 
let $P_{K,n}$ be their convex hull.  For a polygon $\mathcal P$ inside $K$, we let $N(\mathcal P)$ denote the number of vertices 
of $\mathcal P$, and we let $A(\mathcal P)$ denote the area of $K\setminus\mathcal P$.  When one wants to prove central limit 
theorems for $N$ and $A$ in either of the two models of random polygons, it is often the case that 
the Poisson model is easier than the uniform model.  Indeed, in general, results are proved first for the Poisson model, and 
then more arguments are needed to deduce a corresponding result for the uniform model.

Recently, the author \cite{pardonaop} studied the Poisson model of random polygons and proved the following central limit theorem for $N$ and $A$:

\begin{theorem}[\cite{pardonaop}]\label{strongclt}
As $\lambda\to\infty$, the following estimates for $\Pi_{K,\lambda}$ hold uniformly over all $K$ of unit area:
\begin{equation}
\sup_x\left|P\left(\frac{N-\mathbb E[N]}{\sqrt{\Var N}}\leq x\right)-\Phi(x)\right|\ll\frac{\log^2\mathbb E[N]}{\sqrt{\mathbb E[N]}}
\end{equation}
\begin{equation}
\sup_x\left|P\left(\frac{A-\mathbb E[A]}{\sqrt{\Var A}}\leq x\right)-\Phi(x)\right|\ll\frac{\log^2\mathbb E[N]}{\sqrt{\mathbb E[N]}}
\end{equation}
Here $\Phi(x)=P(Z\leq x)$ where $Z$ is the standard normal distribution.
\end{theorem}

In this paper, our goal is to show how to derive the following corollary for the uniform model:

\begin{corollary}\label{uniformclt}
As $n\to\infty$, the following estimates for $P_{K,n}$ hold uniformly over all $K$ of unit area:
\begin{equation}
\sup_x\left|P\left(\frac{N-\mathbb E[N]}{\sqrt{\Var N}}\leq x\right)-\Phi(x)\right|\to 0
\end{equation}
\begin{equation}
\sup_x\left|P\left(\frac{A-\mathbb E[A]}{\sqrt{\Var A}}\leq x\right)-\Phi(x)\right|\to 0
\end{equation}
Here $\Phi(x)=P(Z\leq x)$ where $Z$ is the standard normal distribution.
\end{corollary}

From the estimates derived in this paper, a secondary result from \cite{pardonaop} (Theorem \ref{strongvarest} below) also carries over immediately 
to the uniform model.  We should say that Theorem \ref{strongvarest} and the consequence derived here, 
Corollary \ref{uniformvarest}, have both been proven independently by Imre B\'ar\'any and Matthias Reitzner.

\begin{theorem}[\cite{pardonaop}]\label{strongvarest}
As $\lambda\to\infty$, the following estimates for $\Pi_{K,\lambda}$ hold uniformly over all $K$ of unit area:
\begin{equation}
\mathbb E[N]\asymp\Var N\asymp\lambda\mathbb E[A]\asymp\lambda^2\Var A
\end{equation}
\end{theorem}

\begin{corollary}\label{uniformvarest}
As $n\to\infty$, the following estimates for $P_{K,n}$ hold uniformly over all $K$ of unit area:
\begin{equation}
\mathbb E[N]\asymp\Var N\asymp n\mathbb E[A]\asymp n^2\Var A
\end{equation}
\end{corollary}

Though both Theorem \ref{strongclt} and Corollary \ref{uniformclt} have been known for quite some time in the 
case that either $K$ is a polygon or $\partial K$ is of class $C^2$, the proof given here of Theorem 
\ref{strongclt} $\implies$ Corollary \ref{uniformclt} for \emph{arbitrary} convex $K$ appears to be new.  
Corollary \ref{uniformclt} answers a question of Van Vu \cite{vuopen}.

We will see below that in order to prove that Theorem \ref{strongclt} implies Corollary \ref{uniformclt}, 
it suffices to show that when $n=\lambda$, the random variables $N(P_{K,n})$ and $N(\Pi_{K,\lambda})$ (as well as $A(P_{K,n})$ and $A(\Pi_{K,\lambda})$) 
have the same expectation and variance up to a small enough error.  This is essentially the same strategy 
used by Van Vu in \cite{vu1} to derive a similar implication for a special case of random polytopes.

However, in contrast to \cite{vu1}, we will use relatively down to earth integral geometry to establish our estimates, 
instead of sophisticated arguments from probability theory.  The approach we take is conceptually very simple.  We write down 
integral geometric expressions for the expectation and variance for the Poisson and uniform models, and then estimate their 
difference in the limit $n=\lambda\to\infty$.  In this formulation, the ``reason'' that the desired convergence holds is completely 
transparent: it is essentially reduced to the convergence $(1+\frac\alpha n)^n\to e^\alpha$ as $n\to\infty$.  Also, the variables 
$N$ and $A$ are treated simultaneous with an identical proof for each (c.f. \cite{vu1} where the case of $f_i$, the number of $i$-simplices, 
is harder than the case of the volume and requires a new idea).  An admitted disadvantage 
of this approach is that one has to actually write down these integrals explicitly, however once this is done, no 
further manipulations are necessary.  It is interesting to observe that using integral geometry is almost never the ``right'' way 
to prove statements along the lines of Theorem \ref{strongclt} or even Theorem \ref{strongvarest}, essentially because the expressions quickly become too complicated 
to deal with either conceptually or theoretically.  However for our applications here, the desired estimates become simple 
when written in terms of the integral geometry, so we in fact believe that these integral geometric expressions do in some sense 
give the ``right'' proof of our main lemmas.

One expects that our results and the proofs given here will admit straightfoward generalization to higher dimensions.  For random polytopes in dimension $d\geq 3$, 
Theorems \ref{strongclt}--\ref{uniformvarest} are all known in the case of fixed $K$ whose boundary is $C^2$ and has nonvanishing Gauss curvature, 
due to Reitzner \cite{reitzner1} and Vu \cite{vu1}.  In the case that $K$ is a polytope, the analogue of Theorem \ref{strongclt} was proven very 
recently by B\'ar\'any and Reitzner \cite{baranypolytope} (one expects that an analogue of Theorem \ref{strongvarest} also follows from their methods).  
We expect that if applied to higher dimensions, the methods in this paper would show that Corollaries \ref{uniformclt} and \ref{uniformvarest} follow 
in any situation in which Theorems \ref{strongclt} and \ref{strongvarest} respectively are known to hold (in particular for the case that $K$ is a polytope).  
It is conjectured that Theorems \ref{strongclt}--\ref{uniformvarest} hold for $d\geq 3$ with no restriction on $K$.

\subsection{Acknowledgement}

We thank the referee for useful comments and in particular for making us realize an error in the original proof of Lemma \ref{punifsame}.

\section{Notation and definitions}\label{notdefsec}

We now review some definitions and two basic lemmas from \cite{pardonaop}.

In this paper, $K$ will always denote a (bounded) convex set in $\mathbb R^2$.  Any constants 
implied by the symbols $\ll$, $\gg$, or $\asymp$ are absolute; in particular they are not allowed to depend on $K$.

\begin{figure}[htb]
\centering
\subfigure[Illustration of $W(\theta)$]{\scalebox{.8}{\includegraphics{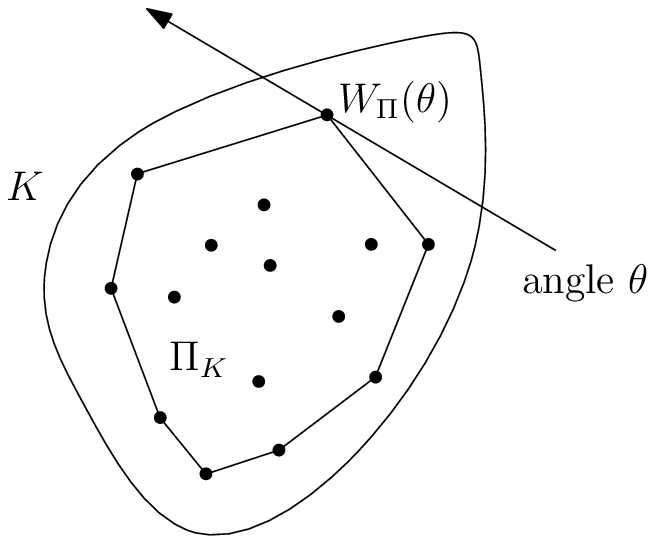}}}
\hspace{0.1in}
\subfigure[Illustration of $C_K(p,\theta)$]{\scalebox{.8}{\includegraphics{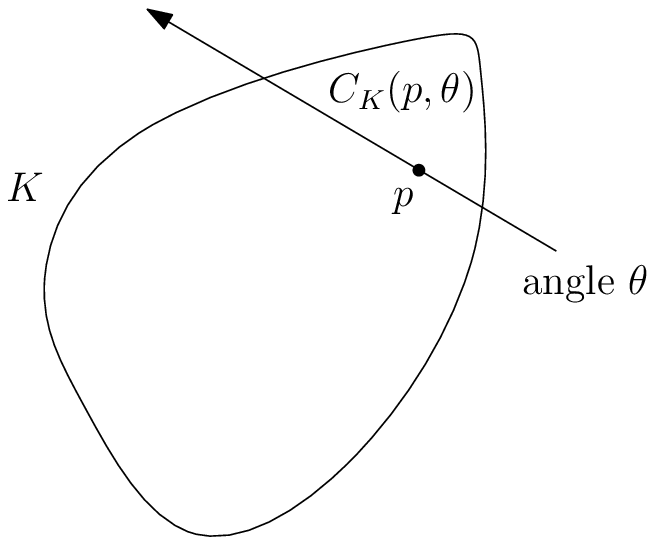}}}
\caption{Illustration of some definitions.}\label{basic}
\end{figure}

Many of the following definitions are 
illustrated in Figure \ref{basic}.  We may leave out the subscript $K$ later when doing so is unambiguous.

\begin{definition}
We define the random variable $W_{\mathcal P}(\theta)$ to be the vertex of $\mathcal P$ which has an oriented tangent line at angle $\theta$.  
This is illustrated in Figure \ref{basic}(a).
\end{definition}

\begin{definition}
A \emph{cap at angle $\theta$} is the intersection of $K$ with a halfplane $H_\theta$ at angle $\theta$.  We 
may specify a cap at angle $\theta$ by giving either its area $r$ or a point $p\in\partial H_\theta$.  
These are denoted $C_K(r,\theta)$ and $C_K(p,\theta)$ respectively; the latter is illustrated in Figure \ref{basic}(b).
\end{definition}

\begin{definition}
We define the real number $A_K(p,\theta)$ to be the area of the cap $C_K(p,\theta)$.
\end{definition}

\begin{lemma}\label{Wdistr}
The random variable $W_{\Pi,\lambda}(\theta)$ has probability distribution given by $\lambda\exp(-\lambda A_K(p,\theta))\,dp$ where $dp$ is the Lebesgue measure.  This 
has total mass $1-e^{-\lambda\operatorname{Area}(K)}$, as $\Pi_{K,\lambda}$ is empty with probability $e^{-\lambda\operatorname{Area}(K)}$.
\end{lemma}

\begin{proof}
This follows directly from the definition of a Poisson point process.  The probability that no point lands in $C_K(p,\theta)$ is $\exp(-\lambda A_K(p,\theta))$, 
and we multiply this by $\lambda\,dp$, which is the density of the Poisson point process.

Alternatively, we may differentiate $\exp(-\lambda A_K(p,\theta))$ with respect to the direction orthogonal to $\theta$ and divide by the length of $\partial H_\theta\cap K$.  
This also yields $\lambda\exp(-\lambda A_K(p,\theta))\,dp$.
\end{proof}

\begin{definition}
We define the function $f_K(x,\theta):[0,1]\times\mathbb R/2\pi\to\mathbb R$ as follows:
\begin{equation}
f_K(x,\theta)=\begin{cases}\text{length of $(\partial H_\theta)\cap K$ where $C_K(\log\frac 1x,\theta)=H_\theta\cap K$}
\cr\hspace{2.2in}\text{if }x>\exp(-\Area(K))\cr 0\hfill\text{if }x\leq\exp(-\Area(K))\end{cases}
\end{equation}
\end{definition}

It will be important to have the following bound on the growth of $f$:

\begin{lemma}\label{boundonF}
If $y\leq x$, then:
\begin{equation}
\frac{f(y)}{\sqrt{-\log y}}\leq\frac{f(x)}{\sqrt{-\log x}}
\end{equation}
\end{lemma}

The bound above is sharp, for instance $f(x)=\operatorname{const}\cdot\sqrt{-\log x}$ for $K=\{x,y\geq 0\}$ (i.e.\ the first quadrant).

\begin{figure}[htb]
\centering
\includegraphics{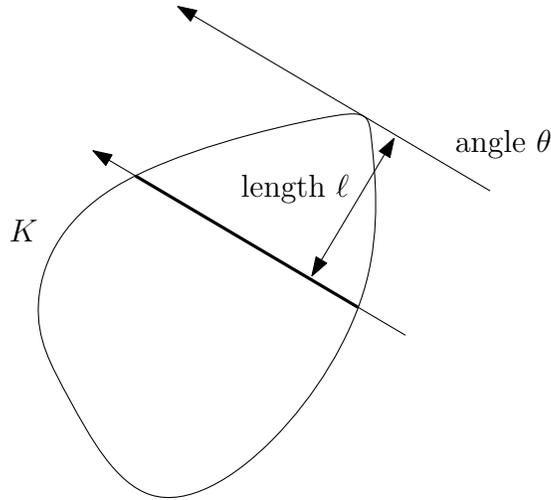}
\caption{Illustration of the function $h$.}\label{figureH}
\end{figure}

\begin{proof}
Project $K$ along the lines at angle $\theta$ to get a height function $h:[0,\infty)\to\mathbb R_{\geq 0}$; 
in Figure \ref{figureH}, $h(\ell)$ is the length of the thick segment.  Now if 
$A(\ell)=\int_0^\ell h(\ell')\,d\ell'$ then $f(\exp(-A(\ell)))=h(\ell)$.  Thus 
we see that it suffices to show that the function:
\begin{equation}
\frac{h(\ell)}{\sqrt{A(\ell)}}
\end{equation}
is decreasing.  Differentiating with respect to $\ell$, we see that it suffices to show that:
\begin{equation}
h(\ell)^2-2h'(\ell)A(\ell)\geq 0
\end{equation}
For $\ell=0$, the left hand side is clearly nonnegative, and the derivative of the 
left hand side equals $-2h''(\ell)A(\ell)$, which is $\geq 0$ by concavity of $h$.
\end{proof}

When proving central limit theorems, it is important to decompose $N$ and $A$ into local pieces.  Thus we define $N(\alpha,\beta)$ to equal the 
number of edges with angle in the interval $(\alpha,\beta)$.  Then it is easy to see that:
\begin{equation}
N=N(\alpha_1,\alpha_2)+N(\alpha_2,\alpha_2)+\cdots+N(\alpha_L,\alpha_1)
\end{equation}
A similar decomposition is valid for $A$, where $A(\alpha,\beta)$ is best explained graphically in Figure \ref{Adecomp}.

\begin{figure}[htb]
\centering
\includegraphics{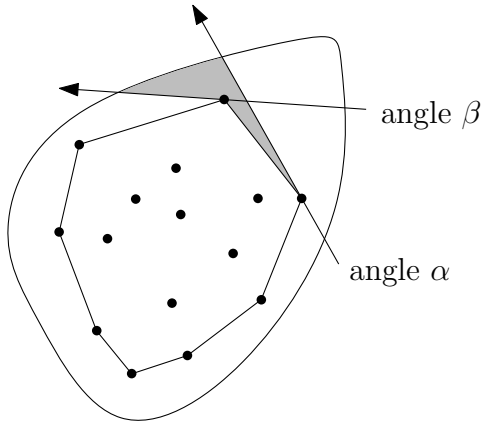}
\caption{Illustration of $A(\alpha,\beta)$}\label{Adecomp}
\end{figure}

Consider for the moment the Poisson model, and let $X$ denote $N$ or $A$.  In \cite{pardonaop}, it is shown that if one chooses 
the partition so that each interval $[\alpha_i,\alpha_{i+1}]$ has constant \emph{affine invariant measure} (a notion from \cite{pardonaop} 
which will not concern us here), then $X(\alpha_i,\alpha_{i+1})$ has constant expectation and variance, and the correlation between 
$X(\alpha_i,\alpha_{i+1})$ and $X(\alpha_j,\alpha_{j+1})$ is exponentially decreasing in $|i-j|$ (specifically, an $\alpha$-mixing 
estimate is proved).  From these facts, along with a general lower bound on the variance of $X$ due to B\'ar\'any and Reitzner 
\cite{baranypreprint} (their result also holds in higher dimensions), it follows on general principles that a central limit 
theorem holds for $X$ in the Poisson model.

\section{Proofs}

Let us begin by stating the lemmas which we will prove.

\begin{lemma}\label{punifsame}
As $n\to\infty$, we have:
\begin{equation}
\sup_x\left|P(A(\Pi_{K,n})\leq x)-P(A(P_{K,n})\leq x)\right|\to 0
\end{equation}
\begin{equation}
\sup_x\left|P(N(\Pi_{K,n})\leq x)-P(N(P_{K,n})\leq x)\right|\to 0
\end{equation}
uniformly over all convex $K$ of unit area.
\end{lemma}

\begin{lemma}\label{expectsame}
As $n\to\infty$, we have:
\begin{equation}
\left|\mathbb E[N(\Pi_{K,n})]-\mathbb E[N(P_{K,n})]\right|=o(\sqrt{\Var N(\Pi_{K,n})})
\end{equation}
\begin{equation}\label{vuproved}
\left|\mathbb E[A(\Pi_{K,n})]-\mathbb E[A(P_{K,n})]\right|=o(\sqrt{\Var A(\Pi_{K,n})})
\end{equation}
uniformly over all convex $K$ of unit area.
\end{lemma}

\begin{lemma}\label{varsame}
As $n\to\infty$, we have:
\begin{equation}
\Var N(\Pi_{K,n})\sim\Var N(P_{K,n})
\end{equation}
\begin{equation}
\Var A(\Pi_{K,n})\sim\Var A(P_{K,n})
\end{equation}
uniformly over all convex $K$ of unit area.
\end{lemma}

Lemma \ref{punifsame} says essentially that as $n\to\infty$, the functionals of $P_{K,n}$ and $\Pi_{K,n}$ have asymptotically 
the same distributions.  Lemmas \ref{expectsame} and \ref{varsame} give us equivalence of the expectation and variance in the 
two models.  It is elementary to observe that these lemmas combine to give Corollary \ref{uniformclt} (and also Corollary \ref{uniformvarest}).  
We note that Van Vu has observed (\ref{vuproved}) in \cite[p224, Proposition 3.1 and Remark 3.5]{vu1}; our proof is different.

\subsection{Proof of Lemma \ref{punifsame}}

\begin{proof}
We follow and slightly correct an argument of Reitzner \cite[p492--3]{reitzner1}.  We thank the referee for asking us to clarify 
our use of Reitzner's argument, since it was this that led us to realize the error.  The argument below becomes fallacious if we 
write $P(X\leq x|E)$ (as Reitzner does) everywhere we have $P(X\leq x\,\&\,E)$.  The problem is that the the first equality in 
equation (\ref{est}) (which is Reitzner's equation (13)) is false in this case (note our $E$ is Reitzner's $A$).  Thus the proof in \cite{reitzner1} is typographically 
very close to being correct; there is no problem once we replace every $P(X\leq x|E)$ with $P(X\leq x\,\&\,E)$.

First, fix $\epsilon>0$ and let $S_\epsilon=\bigcup_{\theta\in\mathbb R/2\pi}C(\theta,\epsilon)$ 
be the union of all caps of area $\epsilon$.  Let $E_P$ and $E_\Pi$ be the events that $\partial P$ and $\partial\Pi$ are completely 
contained in $S_\epsilon$ respectively.  Now trivially, we have that $P(E_P),P(E_\Pi)\to 1$ as $n\to\infty$, uniformly over all $K$ of unit area.

If $B$ is any event, then $|P(B)-P(B\,\&\,E)|\leq 1-P(E)$.  Letting $B$ be $X\leq x$, we have:
\begin{equation}\label{tri1}
\sup_K\sup_x|P(X(\Pi_{K,n})\leq x)-P(X(\Pi_{K,n})\leq x\,\&\,E_\Pi)|\to 0
\end{equation}
\begin{equation}\label{tri2}
\sup_K\sup_x|P(X(P_{K,n})\leq x)-P(X(P_{K,n})\leq x\,\&\,E_P)|\to 0
\end{equation}
as $n\to\infty$, where $X$ denotes either $N$ or $A$.

Now consider $P(X(\Pi_{K,n})\leq x\,\&\,E_\Pi)$ and $P(X(P_{K,n})\leq x\,\&\,E_P)$.  Observe that if we condition 
both probabilities on the number of points of the process in $S_\epsilon$, then they become equal.  Let us call 
this probability $P(X\leq x\,\&\,E|k)$.  In other words, suppose we place $k$ points uniformly at random in $S_\epsilon$.  
Then $P(X\leq x\,\&\,E|k)$ is defined to equal the probability that the boundary of their convex hull is contained 
in $S_\epsilon$ and $X\leq x$.  Thus setting $p=\Area(S_\epsilon)$, we find:
\begin{multline}\label{est}
\left|P(X(\Pi_{K,n})\leq x\,\&\,E_\Pi)-P(X(P_{K,n})\leq x\,\&\,E_P)\right|\cr
=\left|\sum_{k=0}^\infty\frac{(np)^k}{k!}e^{-np}P(X\leq x\,\&\,E|k)-\binom nkp^k(1-p)^{n-k}P(X\leq x\,\&\,E|k)\right|\cr
\leq\sum_{k=0}^\infty\left|\frac{(np)^k}{k!}e^{-np}-\binom nkp^k(1-p)^{n-k}\right|\leq 2p
\end{multline}
where $\binom nk=0$ if $k>n$.  The last bound is due to Vervaat \cite{vervaat}.  Combining (\ref{tri1}) and (\ref{tri2}) with (\ref{est}), we find that:
\begin{equation}
\limsup_{n\to\infty}\sup_K\sup_x|P(X(\Pi_{K,n})\leq x)-P(X(P_{K,n})\leq x)|\leq 2\sup_K\Area(S_\epsilon)
\end{equation}
But we may choose $\epsilon>0$ arbitrarily, so we are done.
\end{proof}

It may indeed be possible to take a similar strategy to prove Lemmas \ref{expectsame} and \ref{varsame}.  However, in this case 
bounding the sum in equation (\ref{est}) becomes harder, since we have expectations instead of probabilities, and the former are 
not bounded by $1$.  Also, proving analogues of equations (\ref{tri1}) and (\ref{tri2}) becomes nontrivial.  Since we need 
good estimates for Lemmas \ref{expectsame} and \ref{varsame}, choosing $\epsilon$ correctly as a function of $n$ and estimating 
$P(E_\Pi)$ and $P(E_P)$ becomes an issue.

\subsection{Proofs of Lemmas \ref{expectsame} and \ref{varsame}}

The proofs of Lemmas \ref{expectsame} and \ref{varsame} will make use of some simple integral geometric expressions for the 
expectations and variances in question.  The derivation of these expressions is completely elementary.  The integrals appear 
complicated, though the point is not their exact form, but rather that they are almost identical for $P_K$ and $\Pi_K$.  With 
the appropriate integrals in hand, the desired convergence essentially reduces to the fact that $(1+\frac\alpha n)^n\to e^\alpha$ 
as $n\to\infty$.  So, before, we begin the proofs, we make some elementary observations about this convergence.  If $e^{-n}<x\leq 1$, 
then $0<1+\frac{\log x}n\leq 1$, so:
\begin{equation}\label{expsmallerconv}
n\log\left(1+\frac{\log x}n\right)\leq n\frac{\log x}n=\log x\implies\frac 1x\left(1+\frac{\log x}n\right)^n\leq 1
\end{equation}
If additionally it holds that $(\log x)^2\leq\frac n2$, then:
\begin{multline}\label{expconvergence}
n\log\left(1+\frac{\log x}n\right)=n\left(\frac{\log x}n+O\left(\frac{(\log x)^2}{n^2}\right)\right)=\log x+O\left(\frac{(\log x)^2}n\right)\cr
\implies\frac 1x\left(1+\frac{\log x}n\right)^n=1+O\left(\frac{(\log x)^2}n\right)
\end{multline}

For the proofs of Lemmas \ref{expectsame} and \ref{varsame}, it is most convenient to use the normalization $\Area(K)=n$ and $\lambda=1$ 
(breaking from our previous convention).  Thus $n$ will be a positive integer, $K$ will have area $n$, and we let $\Pi_K=\Pi_{K,1}$ 
and $P_K=P_{K,n}$.

\begin{proof}[Proof of Lemma \ref{expectsame}]
Let $X$ denote either $N$ or $A$.  In the derivation of the integral geometric expressions, we treat $\Pi_K$ and $P_K$ simultaneously.

The following formula is tautological:
\begin{align}\label{basicExp}
\mathbb E[X]&=\int_{\mathbb R/2\pi}\int_KI_X(p,\theta)\,d\theta\\
I_X(p,\theta)&=\left.\frac d{dh}\mathbb E[X(\theta,\theta+h)|W(\theta)=p]\right|_{h=0}\,dP(W(\theta)=p)
\end{align}
Now let us derive expressions for $I_X(p,\theta)$ for the uniform and Poisson models respectively.  It will 
be convenient to let $y_{\theta,p}$ equal $f(p,\theta)^{-1}$ times the distance from $p$ to $\partial K$ 
in the positive $\theta$ direction.  For every angle $\theta$, the coordinates $x_{p,\theta}:=\exp(-A(p,\theta))$ and $y_{p,\theta}$ 
give a bijection between $K$ and $[e^{-n},1]\times[0,1]$.  It will prove very useful to express points in $K$ 
in terms of these coordinates, mostly because $dP(W_{\Pi}(\theta)=p)=\exp(-A(p,\theta))\,dp=dx\,dy$.

First, let us observe that $\left.\frac d{dh}\mathbb E[X(\theta,\theta+h)|W(\theta)=p]\right|_{h=0}$ is equal to:
\begin{align}\label{incremental1}
\text{For $\Pi_K$ and $X=N$:}\qquad&\frac 12y_{p,\theta}^2f(p,\theta)^2\\
\text{For $P_K$ and $X=N$:}\qquad&\frac 12y_{p,\theta}^2f(p,\theta)^2\frac{n-1}{n-A(p,\theta)}\\
\text{For $\Pi_K$ and $X=A$:}\qquad&\frac 12y_{p,\theta}^2f(p,\theta)^2\\\label{incremental2}
\text{For $P_K$ and $X=A$:}\qquad&\frac 12y_{p,\theta}^2f(p,\theta)^2
\end{align}
And we also observe that $dP(W(\theta)=p)$ equals:
\begin{align}
\text{For $\Pi_K$:}\qquad&\exp(-A(p,\theta))\,dp\\
\text{For $P_K$:}\qquad&\left(1-\frac{A(p,\theta)}n\right)^n\frac n{n-A(p,\theta)}\,dp
\end{align}

Using coordinates $x$ and $y$ in the integral (\ref{basicExp}) and substituting our expressions for $I_X(p,\theta)$, we 
observe that we can integrate out $y_{p,\theta}$ in every case.  The reader can check that the final expressions are:
\begin{align}\label{integralexpect1}
\mathbb E[X]&=\int_{\mathbb R/2\pi}\int_{e^{-n}}^1I_X(x,\theta)\,dx\,d\theta
\end{align}
where $I_X(x,\theta)$ equals:
\begin{align}\label{integralexpect2}
\text{For $\Pi_K$ and $X=N$:}\qquad&\frac 16f(p,\theta)^2\\
\text{For $P_K$ and $X=N$:}\qquad&\frac 16f(p,\theta)^2\frac 1x\left(1+\frac{\log x}n\right)^{n-2}\left(1-\frac 1n\right)\\
\text{For $\Pi_K$ and $X=A$:}\qquad&\frac 16f(p,\theta)^2\\\label{integralexpect3}
\text{For $P_K$ and $X=A$:}\qquad&\frac 16f(p,\theta)^2\frac 1x\left(1+\frac{\log x}n\right)^{n-1}
\end{align}
We now proceed to use the representations (\ref{integralexpect1}) and (\ref{integralexpect2})--(\ref{integralexpect3}) 
to show that the expectations of $X(\Pi_K)$ and $X(P_K)$ are the same up to a relative error of $O(n^{-1+\epsilon})$.

First, observe that our estimate on the growth of $f$ (Lemma \ref{boundonF}) shows that cutting off the integral (\ref{integralexpect1}) 
to $x\geq n^{-B}$ for some large fixed $B$ incurs a relative error of no more than $n^{-B+\epsilon}$.  Now for $x\in[n^{-B},1]$, we 
may use (\ref{expconvergence}) to see that the relative error incurred by replacing $\frac 1x\left(1+\frac{\log x}n\right)^n$ by 
$1$ is no more than $\frac{(\log n)^2}n$.  Observe also that for $x\in[n^{-B},1]$, we know that replacing $1+\frac{\log x}n$ with 
$1$ incurs a relative error of no more than $\frac{\log n}n$.  These operations suffice to transform between the expressions for 
$\mathbb E[X(\Pi_K)]$ and $\mathbb E[X(P_K)]$, so we have shown that they are equal up to a relative error of $O(n^{-1+\epsilon})$.

Thus to finish the proof, we just need to show that:
\begin{equation}
\mathbb E[X(\Pi)]n^{-1+\epsilon}=o(\sqrt{\Var X(\Pi)})
\end{equation}
By a result of B\'ar\'any and Reitzner \cite{baranypreprint}, $\Var X(\Pi)\gg\mathbb E[X(\Pi)]$, so it suffices to show that $\sqrt{\mathbb E[X(\Pi)]}=o(n^{1-\epsilon})$.  
It is trivial to see that $\mathbb E[X(\Pi)]\leq n$, so we are done.
\end{proof}

\begin{proof}[Proof of Lemma \ref{varsame}]
This proof follows the same outline, so we will be a little less explicit; the interested reader can write down the long integrals 
if they so desire.  Again, we let $X$ denote either $N$ or $A$.

Our plan is to show that $\mathbb E[X^2]$ is the same in the two cases up to a relative error of $O(n^{-1+\epsilon})$.

We think of $X$ as being the integral of a random measure $\mu_X$ on $\mathbb R/2\pi$.  This random measure is just 
given by the family of variables $X(\alpha,\beta)$ (explicitly, the measure of the interval $[\alpha,\beta]$ is $X(\alpha,\beta)$).  
Now $X^2$ is just the total mass of $\mu_X\otimes\mu_X$ on $(\mathbb R/2\pi)^2$.  Using linearity of expectation, we just 
need to take the $d\theta\,d\psi$ integral of the expectation of $X(\theta,\theta+d\theta)X(\psi,\psi+d\psi)$.  This 
expectation in turn, we condition on $W(\theta)$ and $W(\psi)$, writing it as an integral $dP((W(\theta),W(\psi))=(p,q))$ over $K\times K$.  
We will often implicitly use the fact that the integrand is \emph{positive}.

The first step is to show that we may remove the region where either $\exp(-A(p,\theta))<n^{-B}$ or $\exp(-A(q,\psi))<n^{-B}$ and 
incur a relative error of $O(n^{-1+\epsilon})$.  By symmetry, let us deal with the region where $\exp(-A(p,\theta))<n^{-B}$.  
Then the contribution to the total integral representing $\mathbb E[X^2]$ is just:
\begin{equation}
\int\limits_{\mathbb R/2\pi}\int\limits_{\{p\in K:A(p,\theta)\geq B\log n\}}\!\!\!\!\!\!\left.\frac d{dh}\mathbb E[X(\theta,\theta+h)\cdot X|W(\theta)=p]\right|_{h=0}\!\!\!\!\,dP(W(\theta)=p)\,d\theta
\end{equation}
Now the $X$ in the expectation contributes at most a multiplicative factor of $n+2$.  With this $X$ removed, the integral 
becomes something we already estimated in the proof of Lemma \ref{expectsame} as being $O(n^{-B+\epsilon})$.  Thus we are done.

The second step is to show that on the region where $A(p,\theta)$ and $A(q,\psi)$ are both $\leq B\log n$, the integrands 
(corresponding to $\mathbb E[X(\theta,\theta+d\theta)X(\psi,\psi+d\psi)]$ in the respective models) are 
equal up to a relative error of $O(n^{-1+\epsilon})$.  As before, this splits up into two problems:

First, we need to 
show that the probability densities $dP((W(\theta),W(\psi))=(p,q))$ in the cases of $\Pi_K$ and $P_K$ are the same up to 
a relative error of $O(n^{-1+\epsilon})$.  As before, we may express $dP((W(\theta),W(\psi))=(p,q))$ elementarily in 
terms of $A(p,q,\theta,\psi):=\Area(C(p,\theta)\cup C(q,\psi))$.  
Then the fact that this quantity is $O(\log n)$ means we may apply (\ref{expconvergence}) to see that the densities 
are equal up to a relative error of $O(n^{-1+\epsilon})$ (note that this is true even for the singular part of the 
measure $dP((W(\theta),W(\psi))=(p,q))$ occurring on the diagonal $p=q$).

Second, we need to show that the incremental expectations 
$\mathbb E[X(\theta,\theta+d\theta)X(\psi,\psi+d\psi)]$ are the same up to a relative error of $O(n^{-1+\epsilon})$.  
Again, this just involves writing equations such as (\ref{incremental1})--(\ref{incremental2}).  Then the fact that 
$A(p,q,\theta,\psi)=O(\log n)$ shows easily that they coincide up to a relative error of $O(n^{-1+\epsilon})$.  Though it presents 
no difficulty in the proof, one should note that when $X=N$, there is a singular component to the measure 
$\mathbb E[X(\theta,\theta+d\theta)X(\psi,\psi+d\psi)]$ on the diagonal $\theta=\psi$.

We have shown that:
\begin{equation}
|\mathbb E[X(\Pi_K)^2]-\mathbb E[X(P_K)^2]|\ll n^{-1+\epsilon}\mathbb E[X(\Pi_K)^2]
\end{equation}
Now $\mathbb E[X(\Pi_K)^2]=\mathbb E[X(\Pi_K)]^2+\Var X(\Pi_K)$.  Thus we have:
\begin{equation}
|\mathbb E[X(\Pi_K)^2]-\mathbb E[X(P_K)^2]|\ll n^{-1+\epsilon}\max(\mathbb E[X(\Pi_K)]^2,\Var X(\Pi_K))
\end{equation}
In the proof of Lemma \ref{expectsame}, we showed that $\mathbb E[X(\Pi_K)]$ and $\mathbb E[X(P_K)]$ are the same up to 
a relative error of $O(n^{-1+\epsilon})$.  This implies then that:
\begin{equation}
|\mathbb E[X(\Pi_K)]^2-\mathbb E[X(P_K)]^2|\ll n^{-1+\epsilon}\mathbb E[X(\Pi_K)]^2
\end{equation}
Thus it follows that:
\begin{equation}
\left|\Var X(\Pi_K)-\Var X(P_K)\right|\ll n^{-1+\epsilon}\max(\mathbb E[X(\Pi_K)]^2,\Var X(\Pi_K))
\end{equation}
Thus to finish the proof, we just need to show that $n^{-1+\epsilon}\mathbb E[X(\Pi)]^2=o(\Var[X(\Pi)])$.  By 
a result of B\'ar\'any and Reitzner \cite{baranypreprint}, $\Var X(\Pi)\gg\mathbb E[X(\Pi)]$, so it suffices to show that $\mathbb E[X(\Pi)]=o(n^{1-\epsilon})$.  
It is a well known estimate (see \cite{pardonaop}) that $\mathbb E[X(\Pi)]\ll n^{1/3}$, so we are done.
\end{proof}

\bibliographystyle{plain}
\bibliography{chull2duniform}

\begin{thebibliography}{1}

\bibitem{vuopen}
{P}roblems from the {AIM} {W}orkshop on {A}lgorithmic {C}onvex {G}eometry.
\newblock {\em Available online at \tt
  http://www.aimath.org/WWN/convexgeometry/convexgeometry.pdf}, 2007.

\bibitem{baranypreprint}
Imre B\'ar\'any and Matthias Reitzner.
\newblock On the variance of random polytopes.
\newblock {\em Adv. Math.}, 225(4):1986--2001, 2010.

\bibitem{baranypolytope}
Imre B{\'a}r{\'a}ny and Matthias Reitzner.
\newblock Poisson polytopes.
\newblock {\em Ann. Probab.}, 38(4):1507--1531, 2010.

\bibitem{pardonaop}
John Pardon.
\newblock Central limit theorems for random polygons in an arbitrary convex
  set.
\newblock {\em To appear in Annals of Probability}, 2010.

\bibitem{reitzner1}
Matthias Reitzner.
\newblock Central limit theorems for random polytopes.
\newblock {\em Probab. Theory Related Fields}, 133(4):483--507, 2005.

\bibitem{vervaat}
W.~Vervaat.
\newblock Upper bounds for the distance in total variation between the binomial
  or negative binomial and the {P}oisson distribution.
\newblock {\em Statistica Neerlandica}, 23:79--86, 1969.

\bibitem{vu1}
Van Vu.
\newblock Central limit theorems for random polytopes in a smooth convex set.
\newblock {\em Adv. Math.}, 207(1):221--243, 2006.

\end{thebibliography}

\end{document}